\documentclass{gtart}

\usepackage{amsmath, amssymb, latexsym, amscd, graphicx, epstopdf}
\usepackage{euscript}
\usepackage{subfigure}
\usepackage{epsfig}
\usepackage{psfrag}

\usepackage[inner=30mm, outer=30mm, textheight=225mm]{geometry}

\def\bkC{{\rm \kern.24em \vrule width.05em height1.4ex depth-.05ex 
\kern-.26em C}}
\def\C{\bkC}
\def\bksC{{\rm \kern.24em \vrule width.05em height1ex depth-.05ex 
\kern-.26em C}}

\def\bkE{{\rm I\kern-.22em E}}
 
\def\bkH{{\rm I\kern-.22em H}}
\def\H{\bkH} 
\def\bkN{{\rm I\kern-.17em N}}
\def\NN{\bkN}
\def\bkQ{{\rm \kern.24em \vrule width.05em height1.4ex depth-.05ex 
\kern-.26em Q}}

\def\bkR{{\rm I\kern-.17em R}}
\def\RR{\bkR}
\def\bkZ{{\rm Z\kern-.32em Z}}

\def\bksZ{{\rm Z\kern-.22em Z}}

\newcommand{\Hthree}{\mathbb{H}^3}
\newcommand{\til}[1]{\widetilde{#1}}
\newcommand{\bdry}{\partial}

\def\PSL{PSL_2(\C)}

\def\D{\mathfrak{D}}

\def\PX{\overline{\mathfrak{X}}}

\def\tri{\mathcal{T}}

\theoremstyle{plain}

\newtheorem{thm}{Theorem}
\newtheorem*{thm*}{Theorem}
\newtheorem{lem}[thm]{Lemma}
\newtheorem*{lem*}{Lemma}
\newtheorem{cor}[thm]{Corollary}
\newtheorem*{cor*}{Corollary}

\newtheorem*{cla*}{Claim}

\newtheorem*{con*}{Condition}

\newtheorem*{pro*}{Proposition}

\newtheorem*{rem*}{Remark}
\newtheorem{defn}[thm]{Definition}
\newtheorem*{defn*}{Definition}
\newtheorem{rmk}[thm]{Remark}
\newtheorem*{rmk*}{Remark}
\newtheorem{obs}[thm]{Observation}

\addtocounter{MaxMatrixCols}{4}
\begin{document}
\title{Pseudo-Developing Maps for Ideal Triangulations I:\\
Essential Edges and Generalised Hyperbolic Gluing Equations} 
\author{Henry Segerman and Stephan Tillmann}
\dedication{To Bus Jaco on the occasion of his seventieth birthday}

\begin{abstract}
Let $N$ be a topologically finite, orientable 3--manifold with ideal triangulation. We show that if there is a solution to the hyperbolic gluing equations, then all edges in the triangulation are essential. This result is extended to a generalisation of the hyperbolic gluing equations, which enables the construction of hyperbolic cone-manifold structures on $N$ with singular locus contained in the 1--skeleton of the triangulation.
\end{abstract}
\primaryclass{57M25, 57N10}
\keywords{3--manifold, ideal triangulation, parameter space, Thurston's gluing equations}


\maketitle


\section{Introduction}

Neumann and Zagier \cite{nz} study the variation of the volume function on a cusped hyperbolic 3--manifold of finite volume using a decomposition of the manifold into hyperbolic ideal tetrahedra. This is based on a construction by Thurston \cite{t}, which associates to the underlying topological ideal triangulation a \emph{parameter space} of shapes of ideal hyperbolic tetrahedra satisfying certain polynomial equations. Such a parameter space can be associated to any ideal triangulation of a non-compact 3-manifold with torus cusps, and several authors have studied this affine algebraic set in this generality (see, for instance, Yoshida \cite{y}, Francaviglia \cite{f}, Tillmann \cite{ti}, Segerman \cite{seg}). A key step in many applications is the construction of a so-called pseudo-developing map for a given point on the parameter space in order to produce a representation of the fundamental group of the manifold into the group of orientation preserving isometries of hyperbolic 3--space. The map is called a pseudo-developing map rather than a developing map because it is not necessarily locally injective.

More recently, Luo \cite{luo} initiated the study of the parameter space for arbitrary ideally triangulated manifolds. In Luo's setting, they arise from (possibly semi-simplicial) triangulations of closed 3-manifolds by removing the vertices. This raises the question of whether the existence of pseudo-developing maps, their associated holonomies and continuous extensions depend on topological hypotheses. This note considers the most general setting, and has the following theorem as its core result. The relevant definitions can be found in Sections~\ref{sec:Definitions} and \ref{sec:defo}.

\begin{thm}\label{thm:essential edges}
Let $N$ be a topologically finite, orientable 3--manifold with ideal triangulation. If there is a solution to the hyperbolic gluing equations, then all edges in the triangulation are essential.
\end{thm}

\begin{cor}\label{cor:essential edges}
Let $N$ be a topologically finite, orientable 3--manifold with ideal triangulation. If there is a solution to the hyperbolic gluing equations, then associated to it are a pseudo-developing map $\widetilde{N}\to \H^3$ and a representation of $\pi_1(N)$ into $\PSL,$ which makes this map equivariant.
\end{cor}

Luo \cite{luo} also introduced \emph{generalised hyperbolic gluing equations}, which state that the product of all shape parameters around an ideal edge is either $+1$ or $-1.$ The sign may be different at different edges. This can be generalised further by choosing, for each edge $e\in N^{(1)},$ an element $\xi_e \in S^1=\{z\in\C \mid |z|=1\}$ and requiring the product of all shape parameters around $e$ to equal $\xi_e.$ The restriction of $\xi_e$ to $S^1$ is natural: if the shape parameters of all tetrahedra have positive imaginary parts, then this gives a (possibly incomplete) hyperbolic cone manifold structure on $N$ with cone angle $\arg(\xi_e)+2\pi k_e$ (for some $k_e \in \NN$) around $e.$ The most interesting case arises when each $\xi_e$ is a root of unity, but we will consider the general case throughout.
These ``$S^1$--valued" gluing equations are determined by the vector $\xi=(\xi_e)=(\xi_e)_{e\in N^{(1)}},$ and are called the \emph{$\xi$--hyperbolic gluing equations}, since they depend on the choice of $\xi.$ Let $o=o(\xi)=(o(\xi_e))_{e\in N^{(1)}},$ where $o(\xi_e) \in \NN \cup \{\infty\}$ is the order of $\xi_e.$  

If one drills out the edges in $N,$ one obtains a handlebody $H$ with a natural epimorphism $\pi_1(H)\twoheadrightarrow\pi_1(N).$ We obtain a representation $\rho\co \pi_1(H) \to \PSL$ for each solution to the $\xi$--hyperbolic gluing equations. If each order $o(\xi_e)$ is finite, then we obtain a natural branched (not necessarily finite) cover $N_o$ of $N$ corresponding to the kernel of $\rho.$ The branch locus is contained in the 1--skeleton, and the cover has branch index $o(\xi_e)$ at $e.$ If some order is infinite, one can still define $N_o,$ but this will have an edge of infinite degree and the points on such an edge of infinite degree are not manifold points. If all orders are equal to $1,$ then $N=N_o.$  Theorem~\ref{thm:essential edges} is thus a special case of the following result:

\begin{thm}\label{thm:branched cover}
Let $N$ be a topologically finite, orientable 3--manifold with ideal triangulation. If there is a solution to the $\xi$--hyperbolic gluing equations for $N,$ then all edges in the induced ideal triangulation of $N_o$ are essential.
\end{thm}

Treating each $\xi_e$ as an additional, circle-valued variable gives the \emph{cone-hyperbolic gluing equations}. The existence of interesting solutions to them comes from the following observation:

\begin{obs}\label{cor:existence}
Given any ideally triangulated, topologically finite, orientable 3--manifold $N,$ there is a complete, non-compact hyperbolic cone-manifold structure on $N$ with singular locus contained in $N^{(1)}$ and having volume exactly the number of tetrahedra in $N$ times the volume of the regular hyperbolic ideal tetrahedron.
\end{obs}


\textbf{Outline:} Triangulations, pseudo-manifolds and essential edges are discussed in Section~\ref{sec:Definitions}. The deformation variety and its generalisations are defined in Section~\ref{sec:defo}, and Yoshida's construction is recalled. The handlebody construction as well as the proofs of the main results can be found in Section~\ref{sec:handlebody}. Examples are given in Section~\ref{sec:zoo}.


\textbf{Acknowledgements:} The authors thank the referee for useful comments. The second author thanks Feng Luo for helpful discussions. This work was supported under the Australian Research Council's Discovery funding scheme (project number DP1095760).


\section{Triangulations with essential edges}
\label{sec:Definitions}


\subsection{Ideal triangulation}

An ideal triangulation $\tri$ of the topologically finite 3-manifold $N$ consists of a pairwise disjoint union of standard Euclidean 3--simplices, $\widetilde{\Delta} = \cup_{k=1}^{n} \widetilde{\Delta}_k,$ together with a collection $\Phi$ of Euclidean isometries between the 2--simplices in $\widetilde{\Delta},$ termed \emph{face pairings}, such that $N = (\widetilde{\Delta} \setminus \widetilde{\Delta}^{(0)} )/ \Phi.$ It is well-known that every non-compact, topologically finite 3--manifold admits an ideal triangulation.

Denote $P = \widetilde{\Delta} / \Phi$ the associated \emph{pseudo-manifold} (or \emph{end-compactification} of $N$) with quotient map $p\co \widetilde{\Delta} \to P.$ Let $\sigma$ be a $k$--simplex in $\widetilde{\Delta}.$ Then $p(\sigma)$ may be a singular $k$--simplex in $P.$ Denote by $P^{(k)}$ the set of all (possibly singular) $k$--simplices in $P.$ An \emph{ideal $k$--simplex} is a $k$--simplex with its vertices removed. The vertices of the $k$--simplex are termed the \emph{ideal vertices} of the ideal $k$--simplex. Similarly for singular simplices. The standard terminology of (ideal) edges, (ideal) faces and (ideal) tetrahedra will be used for the singular simplices in $N$ and $P.$

In this note, it is assumed throughout that $N$ is oriented and that all singular simplices in $N$ are given the induced orientations. It follows that the link of each vertex in $P$ is an orientable surface. 

The case in which each vertex link is a torus and $N$ has a complete, hyperbolic structure supported by the ideal triangulation is the most common setting for the study of Thurston's hyperbolic gluing equations, see \cite{ac, ch, f, nr, nz, y}. The case in which each vertex link is a sphere and $P$ is a closed hyperbolic 3--manifold is treated in \cite{LTY}. This note will not make any of these additional assumption.

It is often convenient to start with the 3--dimensional, closed, orientable pseudo-manifold $P$ with (possibly singular) triangulation $\tri.$ Then $N=P \setminus P^{(0)}$ is an ideally triangulated, non-compact, orientable, topologically finite 3--manifold. 


\subsection{Abstract edge-neighbourhood}

The \emph{degree of an edge} $e$ in $P,$ $\deg(e),$ is the number of 1--simplices in $\widetilde{\Delta}$ which map to $e.$ Given the edge $e$ in $P,$ there is an associated \emph{abstract neighbourhood $B(e)$} of $e.$ This is a ball triangulated by $\deg (e)$ 3--simplices, having a unique interior edge $\widetilde{e},$ and there is a well-defined simplicial quotient map $p_{e}\co B(e)\to P$ taking $\widetilde{e}$ to $e.$ This abstract neighbourhood is obtained as follows.

If $e$ has at most one pre-image in each 3--simplex in $\widetilde{\Delta},$ then $B(e)$ is obtained as the quotient of the collection $\widetilde{\Delta}_{e}$ of all 3--simplices in $\widetilde{\Delta}$ containing a pre-image of $e$ by the set $\Phi_{e}$ of all face pairings in $\Phi$ between faces containing a pre-image of $e.$ There is an obvious quotient map $b_{e}\co B(e)\to P$ which takes into account the remaining identifications on the boundary of $B(e).$

If $e$ has more than one pre-image in some 3--simplex, then multiple copies of this simplex are taken, one for each pre-image. The construction is modified accordingly, so that $B(e)$ again has a unique interior edge and there is a well defined quotient map $b_{e}\co B(e)\to P.$ Complete details can be found in \cite{tillus_normal}, Section 2.3.


\subsection{End-compactification and null-homotopic edges}
\label{sec:Null-homotopic edges}

As above, let $P$ be a 3--dimensional, closed, orientable pseudo-manifold with (possibly singular) triangulation $\tri,$ 
and $N=P \setminus P^{(0)}.$ Let $E = \nu(P^{(0)})$ be an open regular neighbourhood of $P^{(0)}$ in $P,$ chosen in such a way that $\partial E$ meets each singular 3--simplex $\sigma^3$ in $P$ in precisely four pairwise disjoint normal triangles, one at each of its corners. Hence $E\cup \partial E$ has the natural cone structure 
$(\partial E \times [0,1]) / \sim$ where $(x,1)\sim(y,1)$ if $x,y$ lie on a connected component of $\partial E.$

Then $C = P \setminus E$ is termed a \emph{compact core} of $N.$ It follows that $P$ can be viewed as obtained from $C$ by taking each connected component $B$ of $\partial C$ and either collapsing it to a point or by attaching the cone over $B$ to a point. Neumann and Yang \cite{ny} call this the \emph{end-compactification} of $N.$ We have 
$$C \subset N \subset P.$$
Let $\widetilde{P}$ be the space obtained from the universal cover $\widetilde{C}$ of $C$ by attaching the cone over each connected boundary component to a point. We then have natural inclusions
$$\widetilde{C} \subset \widetilde{N} \subset \widetilde{P},$$
and $\widetilde{P}$ is termed the \emph{end-compactification of $\widetilde{N}$ with respect to $N.$} Note that $\widetilde{P}$ is also simply connected since adding cones over connected spaces does not increase the fundamental group.

It is hoped that the notation and terminology does not lead to confusion. For instance, when $N$ is hyperbolic, then $\widetilde{N}$ is an open ball and the natural compactification of this open ball (without reference to $N$) is homeomorphic to the 3--ball, whilst $\widetilde{P}$ is $\widetilde{N}$ with countably many points added. Also, there are many examples where $P$ is simply connected; for instance if $C$ is the complement of a knot or link in $S^3.$

The space $\widetilde{P}$ has a natural decomposition into 3--simplices coming from the decomposition of $\widetilde{C}$ into truncated 3--simplices and the coning construction. Lifting the ideal triangulation of $N$ to $\widetilde{N}$ gives a natural decomposition into ideal 3--simplices. It follows from the construction that the ideal triangulation of $\widetilde{N}$ is precisely the restriction of the triangulation of $\widetilde{P}$ to $\widetilde{P}\setminus\widetilde{P}^{(0)}.$ In particular, we have a well-defined simplicial map $\widetilde{P} \to P.$

A triangulation $\tri$ of $P$ is said to be \emph{almost non-singular} if no 3--simplex has two of its edges identified. In this case, the only self-identifications of a 3--simplex are at the vertices. A triangulation $\tri$ of $P$ is \emph{non-singular} if no 3--simplex has any self-identifications. 

A triangulation $\tri$ of $P$ is said to be \emph{virtually almost non-singular} if the induced triangulation of $\til{P}$ is almost non-singular, and \emph{virtually non-singular} if the induced triangulation of $\til{P}$ is non-singular.

Note that a triangulation $\tri$ of $P$ is virtually almost non-singular if and only if every ideal 3--simplex in $\widetilde{N}$ is embedded.

An edge $e \in P^{(1)}$ is \emph{null-homotopic} if and only if there is a map $f \co D^2 \to P$ such that $f(\partial D^2) = e.$ 
Equivalently, $e$ is null-homotopic in $P$ if and only if it represents the trivial element in $\pi_1(P).$ 

The intersection $\alpha = e\cap C$ is \emph{homotopic into $\partial C$} if and only if there is an arc $\beta \subset \partial C$ such that $\alpha$ is homotopic to $\beta$ by a fixed-endpoint homotopy.


To simplify terminology, we will say that the edge $e$ in $P$ is \emph{essential} if $e\cap C$ is not homotopic into $\partial C,$ and it is \emph{not essential} otherwise. This reflects standard terminology for ideal edges in ideal triangulations.
\begin{lem}\label{lem: virt non sing iff essential}
$\tri$ is virtually non-singular if and only if every edge in $P$ is essential.
\end{lem}
\begin{proof}
Suppose the edge $e$ is not essential. Then $e\cap C$ is homotopic into $\partial C.$ The disc in $C$ lifts to a disc in $\widetilde{C}$ with part of its boundary on a lift $\tilde{e}$ of $e$ and the remainder of its boundary on a connected component of $\partial \widetilde{C}.$ Hence $\tilde{e}$ is an edge in $\widetilde{P}$ with both end-points at the same vertex. Thus, any tetrahedron containing $\tilde{e}$ is not embedded and therefore $\tri$ is not virtually non-singular.

Conversely, suppose $\tri$ is not virtually non-singular. There is a tetrahedron in $\widetilde{P}^{(3)}$ with self-identifications and hence an edge $e$ with both end-points at the same vertex. The end-points of $\alpha = e \cap \widetilde{C}$ lie on the same boundary component of $\widetilde{C},$ and hence can be connected by an arc $\beta \subset \partial \widetilde{C}.$ Since $\widetilde{C}$ is simply connected, the loop $\alpha \cup \beta$ bounds an immersed disc in $\widetilde{C}.$ Let $p_C \co \widetilde{C} \to C$ be the covering map. It follows that $p_C(\alpha)$ is homotopic to $p_C(\beta),$ and hence the edge containing $p_C(\alpha)$ is not essential.
\end{proof}

\begin{lem}\label{not_ess_implies_null_htpc}
If $e$ is not essential in $P,$ then $e$ is \emph{null-homotopic} in $P.$
\end{lem}

\begin{proof}
If $\alpha = e\cap C$ is homotopic into $\partial  C,$ then both of its end-points lie on the same boundary component, $B,$ of $C.$ Hence the map $f \co D^2 \to C$ can be extended to a continuous map $f_P \co D^2 \to P$ by coning it over $\beta$ to the vertex of $P$ corresponding to $B.$
\end{proof}

The converse of Lemma \ref{not_ess_implies_null_htpc} is not true. For instance the ideal triangulation of a knot or link in $S^3$ gives rise to a simply connected pseudo-manifold, but there are many edges which are not homotopic into the boundary. For instance, Thurston's ideal triangulation of the figure eight knot complement yields a simply connected pseudo-manifold $P$ with the property that every edge is null-homotopic and essential.


\section{The deformation variety and its friends}
\label{sec:defo}


\subsection{Deformation variety}
\label{sec:Deformation variety}

Let $\Delta^3$ be the standard 3--simplex with a chosen orientation. Suppose the edges from one vertex of $\Delta^3$ are labeled by $z,$ $z'$ and $z''$ so that the opposite edges have the same labeling. Then the cyclic order of $z,$ $z'$ and $z''$ viewed from each vertex depends only on the orientation of the 3--simplex. It follows that, up to orientation preserving symmetries, there are two possible labelings, and we fix one of these labelings. The labels are termed \emph{shape parameters}.

Suppose $P^{(3)} = \{ \sigma_1, \ldots, \sigma_n\}.$ For each $\sigma_i \in P^{(3)},$ fix an orientation preserving simplicial map $f_i\co \Delta^3 \to \sigma_i.$ Let $P^{(1)} = \{ e_1, \ldots, e_m\},$ and let $a^{(k)}_{ij}$ be the number of edges in $f_i^{-1}(e_j),$ which have label $z^{(k)}.$

For each $i \in \{1,\ldots, n\},$ define
\begin{equation}\label{eq:para}
    p_i   = z_i (1 - z''_i) - 1,\quad
    p'_i  = z'_i (1 - z_i) - 1,\quad
    p''_i = z''_i (1 - z'_i) - 1,
\end{equation}
and for each $j \in \{1,\ldots, m\},$ let
\begin{equation}\label{eq:glue}
    g_j = \prod_{i=1}^n z_i^{a_{ij}} {(z'_i)}^{a'_{ij}} {(z''_i)}^{a''_{ij}}
    -1. 
\end{equation}
Setting $p_i=p'_i=p''_i = 0$ gives the \emph{parameter relations}, and setting $g_j=0$  gives the \emph{hyperbolic gluing equations}. For a discussion and geometric interpretation of these equations, see \cite{t, nz}. The parameter relations imply that $z_i \neq 0,1.$
\begin{defn}
The \emph{deformation variety $\D (\tri)$} is the variety in $(\C\setminus \{0\})^{3n}$ defined by the hyperbolic gluing equations together with the
parameter relations.
\end{defn}

Thurston's original parameter space is obtained by choosing a coordinate from each coordinate triple in the above. In calculations, we will often use such a smaller coordinate system. Theorem~\ref{thm:essential edges} is equivalent to the statement that $\D (\tri)\neq \emptyset$ implies that all edges are essential.


\subsection{Yoshida's construction}
\label{sec: developing}

Let $C \subset N \subset P,$ $\tri$ and $\D(\tri)$ be as defined above. Given $Z \in \D (\tri),$ each ideal tetrahedron in $N$ has edge labels which can be lifted equivariantly to $\widetilde{N}.$ Following \cite{y}, we would like to define a continuous map $\Phi_Z\co \widetilde{N} \to \H^3,$ which maps every ideal tetrahedron $\sigma$ in $\widetilde{N}$ to an ideal hyperbolic 3--simplex $\Delta(\sigma),$ such that the labels carried forward to the edges of $\Delta(\sigma)$ correspond to the shape parameters of $\Delta(\sigma)$ determined by its hyperbolic structure; see \cite{t} for the geometry of hyperbolic ideal tetrahedra. Thus, 
\begin{center}
\emph{we need to assume that $\tri$ is virtually almost non-singular}, 
\end{center}
since no ideal hyperbolic simplex has edges identified with each other. It will be necessary to have a consistent choice of parameterisation for the maps from ideal tetrahedra in $\widetilde{N}$ to $\H^3;$ following Thurston \cite{t}, we ensure this by assuming that the map $\sigma \to \H^3$ always is a \emph{straight map}. See \cite{t1, f, LTY} for the details, which play no role in the following.

Each ideal 3--simplex in $\widetilde{N}$ inherits edge labels from $Z.$ Choose a tetrahedron $\sigma$ in $\widetilde{N}$ and an embedding of $\sigma$ into $\H^3$ of the specified shape. For each tetrahedron of $\widetilde{N}$ distinct from $\sigma$ and which has a face in common with $\sigma,$ there is a unique embedding into $\H^3$ which coincides with the embedding of $\sigma$ on the common face and which has the shape determined by $Z.$ If the shape parameters of the new tetrahedron give an ideal hyperbolic tetrahedron of orientation opposite to that of the first one, then this map is not locally injective along the common face. Hence the map is called a \emph{pseudo-developing map} rather than a developing map.

\begin{lem}
Starting with an embedding of $\sigma \subset \widetilde{N},$ there is a unique way to extend this to a well-defined, continuous map $\Phi_Z \co \widetilde{N} \to \H^3,$ such that each ideal tetrahedron in $\widetilde{N}$ is mapped to a hyperbolic ideal tetrahedron of the specified shape. 
\end{lem}

\begin{proof} 
Since $\widetilde{N}$ is simply connected and each abstract edge neighbourhood is embedded, it follows from the hyperbolic gluing equations that the map is well-defined. The reader may find pleasure in doing this exercise or consult \cite{ac} for a full treatment.
\end{proof}

\begin{lem}\label{lem:special case}
If $\D(\tri)\neq \emptyset$ and $\tri$ is virtually almost non-singular, then $\tri$ is virtually non-singular (hence all edges in $P$ are essential).
\end{lem}

\begin{proof}
Let $Z \in \D(\tri).$ The Yoshida map $\Phi_Z$ is well defined since $\tri$ is virtually almost non-singular. Suppose the ideal tetrahedron $\sigma$ in $\widetilde{N}$ has two ideal vertices at the same vertex of $\widetilde{P}.$ Let $t$ and $t'$ denote the normal triangles in $\sigma \cap \bdry \widetilde{C}$ dual to these vertices. Since $t$ and $t'$ are in the link of the same vertex of $\widetilde{P},$ there is a path in $\bdry \widetilde{C}$ from $t$ to $t'$ passing through finitely many normal triangles. This path corresponds to a finite sequence $\sigma_1, \ldots, \sigma_k$ of ideal tetrahedra in $\widetilde{N},$ and a corresponding sequence $v_1, \ldots, v_k$ of ideal vertices of the $\sigma_i$, dual to the sequence of normal triangles. The map $\Phi_Z$ extends to map these ideal vertices into $\overline{\H}^3$, and as we develop along the sequence of tetrahedra, we see that $\Phi_Z(v_1) = \Phi_Z(v_2) = \cdots = \Phi_Z(v_k)$.
But $v_1$ and $v_k$ are two distinct vertices of $\sigma$, so the image of $\sigma$ cannot be an ideal hyperbolic 3--simplex, and thus the Yoshida map is not well-defined.
\end{proof}

\begin{cor}\label{cor:yoshida depends on ideal points}
If the Yoshida map $\Phi_Z\co \widetilde{N} \to \H^3$ is well-defined, then it extends to a continuous map $\overline{\Phi}_Z \co \widetilde{P} \to \overline{\H}^3.$
\end{cor}


\subsection{Representations}
\label{sec:reps}

Suppose that the Yoshida map is well-defined and, in particular, that $\tri$ is virtually almost non-singular. We will see (in Corollary \ref{cor:essential edges}) that these hypotheses hold if $\D (\tri)\neq \emptyset$.

For each $Z \in \D (\tri),$ the Yoshida map $\Phi_Z$ can be used to define a representation $\rho_Z \co \pi_1(N) \to \PSL$ as follows (see \cite{y}). A representation into $\PSL$ is an action on $\H^3,$ and this is the unique representation which makes $\Phi_Z$ $\pi_1(N)$--equivariant: $\Phi_Z(\gamma \cdot x) = \rho_Z(\gamma) \Phi_Z(x)$ for all $x \in \widetilde{N},$ $\gamma \in \pi_1(N).$  Thus, $\rho_Z$ is well--defined up to conjugation, since it only depends upon the choice of the embedding of the initial tetrahedron $\sigma.$ This yields a well--defined map $\chi_{\tri}\co \D (\tri) \to \PX (N)$ from the deformation variety to the $\PSL$--character variety. It is implicit in \cite{nr} that $\chi_{\tri}$ is algebraic; see \cite{ac} for details using a faithful representation of $\PSL \to SL(3, \C).$ Note that the image of each peripheral subgroup under $\rho_Z$ has at least one fixed point on the sphere at infinity.

\begin{rmk}
The representation associated to a solution of the hyperbolic gluing equations may be reducible, or even trivial. For instance, the triangulation of $S^3$ with two tetrahedra obtained by identifying the boundary spheres of the tetrahedra in the natural way has a curve of solutions to the hyperbolic gluing equations, and the associated representations of the fundamental group of $S^3$ minus four points are all trivial.
\end{rmk}

The representation arising from Yoshida's construction can be understood using \emph{elementary face pairings} as follows. An elementary face pairing of the hyperbolic ideal tetrahedron $\Delta$ is an element of $\PSL$ taking one face of $\Delta$ to another. If the deck transformation $\gamma \in \pi_1(N)$ takes the ideal triangle $\tau$ in $\til{N}$ to the ideal triangle $\gamma \cdot \tau,$ and $\sigma^3_0, \ldots, \sigma^3_k$ is a sequence of tetrahedra with the property that $\tau \subset \sigma^3_0,$ $\gamma\cdot\tau \subset \sigma^3_k,$ and consecutive tetrahedra share a face distinct from $\tau$ and $\gamma\cdot\tau,$ then there is an associated product of elementary face pairings, one for each $\Phi_Z(\sigma^3_j),$ such that $\rho_Z(\gamma)$ is their product. This fact is used in \cite{ac} to show that $\chi_{\tri}$ is algebraic; it will be used below for a different purpose.


\subsection{Generalisations}
\label{sec:generalisations}

For each $e_j \in P^{(1)},$ we introduce a variable $\xi_{e_j} \in S^1,$ and define
\begin{equation}\label{eq:glue}
    g_{e_j}(\xi_{e_j}) = \prod_{i=1}^n z_i^{a_{ij}} {(z'_i)}^{a'_{ij}} {(z''_i)}^{a''_{ij}}
    -\xi_{e_j}. 
\end{equation}
Setting $g_{e_j}(\xi_{e_j})=0$ gives the \emph{cone-hyperbolic gluing equation} for $e_j.$ Since the product of all shape parameters associated to all edges equals $1,$ we have the following consequence:
$$\prod_{e \in P^{(1)}}\xi_e =1.$$ 
We also define the \emph{holonomy around $e_j$} to be
$$h(e_j) = \prod_{i=1}^n z_i^{a_{ij}} {(z'_i)}^{a'_{ij}} {(z''_i)}^{a''_{ij}}.$$

\begin{defn}
For a triangulation with $n$ tetrahedra and $m$ edges, the \emph{cone-deformation variety $\D (\tri; \star)$} is the variety in $(\C-\{0 \})^{3n}\times (S^1)^m$ defined by the cone-hyperbolic gluing equations together with the parameter relations. 
\end{defn}

The cone-deformation variety is non-empty for \emph{any} triangulation. Indeed, if $z=\frac{1}{2}(1 + \sqrt{-3}),$ then all of $z,$ $z'$ and $z''$ are roots of unity and specify the regular hyperbolic ideal 3--simplex. This is the unique hyperbolic ideal 3--simplex with the property that all shape parameters are roots of unity. Assigning values to all shape parameters in this way, one can solve $g_e(\xi_e)=0$ for $\xi_e \in S^1$ uniquely for each $e.$ It turns out that this solution has particularly nice properties, giving us Observation~\ref{cor:existence}.


\begin{proof}[Proof of Observation~\ref{cor:existence}]
Identify each ideal 3--simplex in $\widetilde{\Delta} \setminus \widetilde{\Delta}^{(0)}$ with the regular ideal 3--simplex. Gluing faces in pairs by hyperbolic isometries according to the face pairings gives a hyperbolic structure on $N$ minus edges. Around the edges there is no shearing, and the total angle around edge $e$ is its degree times $\frac{\pi}{3}.$ Hence there is a non-compact hyperbolic cone-manifold structure on $N$ with singular locus contained in $N^{(1)}.$ Near the ideal vertices of the tetrahedra, one can consistently choose horospherical triangles that match up and the structure is therefore complete.
\end{proof}

Denote the projections to the factors $F_1 \co \D (\tri; \star) \to (\C-\{0 \})^{3n}$ and $F_2\co \D (\tri; \star) \to (S^1)^m.$ Given $\xi \in (S^1)^m,$ we let $\D (\tri; \xi)= F_2^{-1}(\xi).$ Then $F_2^{-1}((\xi_1, \ldots, \xi_m))$ may be empty, and $\D(\tri) = \D(\tri, (1, \ldots, 1)).$


\section{The handlebody construction}
\label{sec:handlebody}


\subsection{The handlebody cover}

Consider the compact manifold $H := C \setminus \nu(P^{(1)})$, where $\{\nu(e)\mid e\in P^{(1)}\}$ is a set of pairwise disjoint open tubular neighbourhoods of the edges $P^{(1)}$. The manifold $H$ is a handlebody, and inherits a cell decomposition from $\tri$ into \emph{doubly truncated tetrahedra}: truncated at the vertices and at the edges. See Figure~\ref{trunc_tetra.pdf} for a picture of a doubly truncated tetrahedron. A doubly truncated tetrahedron has four \emph{boundary hexagonal faces} on $\bdry C$, six \emph{rectangular faces} on $\bdry H \setminus \bdry C,$ and four \emph{interior hexagonal faces}.

Consider the decomposition of $\bdry H$ into the union of the \emph{vertex boundary} $\bdry H \cap \bdry C$ (composed of boundary hexagonal faces) and the \emph{edge boundary} $\bdry H \setminus (\bdry H \cap \bdry C$) (composed of rectangular faces). The edge boundary consists of a pairwise disjoint union of annulus components, one for each edge of $C$.

For any cover $K$ of $H$, define a topological space $K^*$ as follows. Lift the decomposition of $H$ to $K$. The boundary of $K$ decomposes into \emph{vertex boundary} components made up out of boundary hexagonal pieces, and \emph{edge boundary} components made up out of rectangular faces. For each edge boundary component of $H,$ we fix a product structure which identifies it with $S^1 \times [0,1].$ Then the edge boundary components of $K$ are either of the form $S^1 \times [0,1]$ or $\mathbb{R} \times [0,1],$ and the product structure is preserved by the deck transformations. We form $K^*$ by first collapsing each edge boundary component of $\bdry K$ by projection to the $[0,1]$ factor, to form a topological space $K'$. Next we collapse each component of the boundary of $K'$ to form $K^*$. 

As in the construction of $P$ from $C$, we can equivalently construct $K^*$ from $K$ by coning rather than collapsing. Note that $H^* = P$. Also note that $K^*$ is a union of tetrahedra identified along faces. Unlike $P$ or $\til{P}$, $K^*$ can have edges incident with infinitely many tetrahedra. We will be interested in the universal cover $\til{H}$ of $H,$ and in $\til{H}^*$. Since $H$ is a handlebody, $\til{H}$ is homotopy equivalent to a tree. Notice that the construction gives a natural quotient map $\til{H}^* \to P.$

\begin{lem}\label{lem: non-singular in handlebody cover}
Every tetrahedron in $\til{H}^*$ is non-singular.
\end{lem}

\begin{proof}
If there is a tetrahedron with self-identifications of any kind then there will be some pair of its vertices that are identified, and thus an edge with its endpoints being the same point.  


Suppose that we have such an edge $e$. Truncate all of the tetrahedra of $\til{H}^*$ at the vertices to form a topological space $\til{H}^*_C$ (a ``core'' of $\til{H}^*$, analogous to the compact core $C$ of $P$). Both points $p_1,p_2 \in e \cap \bdry \til{H}^*_C$ are on the same component of $\bdry \til{H}^*_C$, so we can choose a path $\beta$ contained in a single component of $\bdry \til{H}^*_C$ whose endpoints are $p_1$ and $p_2$. We may deform $\beta$ slightly so that it does not pass through any edges other than $e$. We can now drill out the neighbourhoods $\nu(e_i)$ of the edges of $\til{H}^*_C$ to get back $\til{H}$, with the curve $\beta$ contained in a component of $\bdry \til{H} \cap \bdry \til{C}$, and meeting $\bdry \til{H} \setminus \bdry \til{C}$ only at its endpoints, which are on the edge boundary near the two ends of $e$.

We may assume that the path $\beta$ starts in some (doubly truncated) tetrahedron $\sigma_0$ at an intersection between a boundary hexagonal face and a rectangular face, traverses through the tetrahedra of $\til{H}$ along $\bdry \til{H} \cap \bdry \til{C}$, and then returns back to $\sigma_0$, ending at the other end of the rectangular face. 

Consider such a path $\beta$ that visits a minimal number of tetrahedra. Since the tetrahedra form a tree, there must be at least one ``leaf'' tetrahedron $\sigma$ in the path. That is, the path enters $\sigma$ from one face gluing (at the interior hexagonal face $f$) and exits at the same face gluing. The path $\beta$ is restricted to lie in  $\bdry \til{H} \cap \bdry \til{C}$, and so it enters $\sigma$ at the edge of one of the three boundary hexagonal faces of $\sigma$ adjacent to $f$. There is no path within $\sigma$ along  $\bdry \til{H} \cap \bdry \til{C}$ from one of the boundary hexagonal faces to any of the others, so it must exit $\sigma$ at the same edge. But then the part of $\beta$ in $\sigma$ could be homotoped away, and the path was not minimal. This gives a contradiction.\end{proof}


\subsection{Representations of the handlebody group}

Since there are no gluing consistency conditions to satisfy for the tetrahedra in $\til{H}^*$ (because every edge of $\til{H}^*$ is of infinite degree), we can assign \emph{arbitrary} shape parameters $z\in\C\setminus\{0,1\}$ to the tetrahedra of $\til{H}^*$ and define a pseudo-developing map $D\co\til{H}^*\rightarrow\overline{\H}^3.$ This follows as in the proof of Lemma~\ref{lem: non-singular in handlebody cover} from the fact that $\til{H}^*$ is homotopy equivalent to a tree. In particular, we can build the developing map by starting with a tetrahedron $\sigma_0$ of $\til{H}^*$ and any ideal hyperbolic tetrahedron in $\H^3$, and then developing along any non-backtracking path of tetrahedra in $\til{H}^*$ starting with $\sigma_0$, using any non-degenerate shapes of ideal hyperbolic tetrahedra. As we develop, we never have any consistency conditions to satisfy, because the tetrahedra of $\til{H}^*$ form a tree.

\begin{lem}
Suppose that the parameters for the tetrahedra in $\til{H}^*$ are lifts of the parameters of the tetrahedra of $H.$ Then there exists a well-defined representation $\rho\co\pi_1H\to \PSL.$ Moreover, this representation makes $D\co\til{H}^*\rightarrow\overline{\H}^3$ equivariant if and only if for each edge $e$ in $P,$ the holonomy around $e$, $h(e)$ is an element of $S^1.$ 
\end{lem}

\begin{proof}
The map $D$ has been defined as in the Yoshida construction, and we would like to define $\rho$ in a similar fashion. Consider a triangle $\tau_0$ of $\til{H}^*$, and its image $\gamma \cdot \tau_0$ under the deck transformation $\gamma \in \pi_1H$. The image of the three vertices of $\tau_0$ under $D$ give a triplet of distinct points on $\bdry \H^3$, and the three vertices of $\gamma \cdot \tau_0$ give another triplet. We define $\rho(\gamma)$ to be the unique element of $\PSL$ which maps the first triplet to the second. Since the parameters for the tetrahedra in $\til{H}^*$ are lifts of the parameters of the tetrahedra of $H$, this definition is independent of the choice of $\tau_0.$ This proves the existence of $\rho.$

It follows from the construction that $D(\gamma \cdot x) = \rho(\gamma) D(x)$ for all $x \in \til{H}^*$ and $\gamma \in \pi_1(H)$ \emph{except possibly} for those $x$ that are contained in the 1--skeleton. In $\til{H},$ we have a product structure on the lifts of the edge boundary components of $H,$ which is preserved by the deck transformations. Consider a deck transformation $\gamma\in \pi_1 H,$ which preserves the edge boundary component $B$ of $\til{H}.$ Now in $\til{H}^*, B$ is mapped to a 1--simplex, $e',$ and $l=D(e')$ is a geodesic in $\H^3.$ Composing $D$ with an isometry of $\H^3,$ we may assume that $l=[0,\infty].$ Since $\gamma$ preserves the product structure, $\gamma \cdot x = x$ for each $x \in e'$. Thus $D$ is equivariant with respect to $\rho$ if and only if $\rho(\gamma)$ acts on $\H^3$ by fixing $l$ pointwise, i.e.\thinspace acts as a (possibly trivial) rotation about $l.$ 

In $\til{H}^*,$ $\gamma$ acts as a translation on the set of all 3--simplices incident with $e'.$ A connected fundamental domain for this action consists of a finite number of 3--simplices meeting in $e',$ and their number equals the degree of the corresponding edge $e$ in $P.$ Let $\sigma$ be a 3-simplex in this fundamental domain. There is a unique isometry taking $D(\sigma)$ to $D(\gamma\cdot\sigma),$ and this is the rotation with eigenvalue precisely the product $h(e)$ of all shape parameters at $e$ in $P.$ Hence $\rho(\gamma)$ is a rotation if and only if $h(e)$ is an element in $S^1.$
\end{proof}

\begin{defn}\label{defn:cone-manifold cover}
Given $\til{H}^* \to P,$ $D\co\til{H}^*\rightarrow\overline{\H}^3$ and $\rho\co\pi_1H\to \PSL$ as above, let $P_o=\til{H}^*/\ker(\rho)$ and $N_o = P_o\setminus P_o^{(0)}.$ Since $\pi_1(H)$ acts simplicially, $P_o$ and $N_o$ have natural decompositions into simplices.
\end{defn}

If $o(\xi_e)$ is finite for each $e \in N^{(1)},$ then each edge in $N_o$ has finitely many 3--simplices incident with it; namely if $N_o^{(1)}\ni \tilde{e}\to e \in N^{(1)},$ then $\deg(\tilde{e})= o(\xi_e)\deg(e).$ In this case, the natural map $N_o \to N$ is a (not necessarily finite) branched cover, with branch locus contained in the 1--skeleton and group of deck transformations isomorphic to $\pi_1H/\ker(\rho).$ If $o(\xi_e)$ is infinite for some $e,$ then the points of $N_o$ mapping to $e$ are not manifold points.


\begin{proof}[Proofs of Theorems~\ref{thm:essential edges} and \ref{thm:branched cover}]

We first give a proof of Theorem~\ref{thm:essential edges}, and then modify it for the general case. Note that the inclusion map $H \hookrightarrow C$ induces an epimorphism $\pi_1H\twoheadrightarrow\pi_1C$. The kernel of this map is generated by certain loops around the annuli in  $\bdry H \setminus (\bdry H \cap \bdry C)$. For each edge $e$ we denote by $\gamma_e$ a loop around the annulus corresponding to $e$, where $\gamma_e$ is a level set in the product structure, with an arbitrary orientation.

Given Lemma~\ref{lem:special case}, it suffices to assume for contradiction that there is a solution $Z\in \D(\tri)$, but that $\tri$ is not virtually almost non-singular. In particular, by Lemma \ref{lem: virt non sing iff essential}, some edge is inessential. We work with the compact core $C$ of $N$.

Give the tetrahedra in $\til{H}$ the shape parameters inherited from $Z\in \mathfrak{D}(\tri)$ (which are first inherited by $H,$ then lifted to $\til{H}$ and finally inherited by $\til{H}^*$). Associated are a developing map $D:\til{H}^*\rightarrow \H^3$ and a representation $\rho\co \pi_1(H) \to \PSL.$ Recall that we specify a quotient map $\til{H} \rightarrow \til{H}^*$ by collapsing edge boundary components $\mathbb{R} \times [0,1]$ to the second factor, then collapsing each vertex boundary component.

Now let $e$ be an inessential edge in $N$. Consider a path $\alpha':[0,1]\rightarrow H$, which is the core curve of a rectangular face of a doubly truncated tetrahedron incident to the edge boundary annulus corresponding to $e$. Let $\alpha:[0,1]\rightarrow H$ be the result of pushing $\alpha'$ slightly off the rectangle into the tetrahedron, keeping the endpoints on the boundary hexagonal faces. Then $\alpha$ is parallel to $e$, as in Figure~\ref{trunc_tetra.pdf}. 

 \begin{figure}[ht]
\centering
\includegraphics[width=0.3\textwidth]{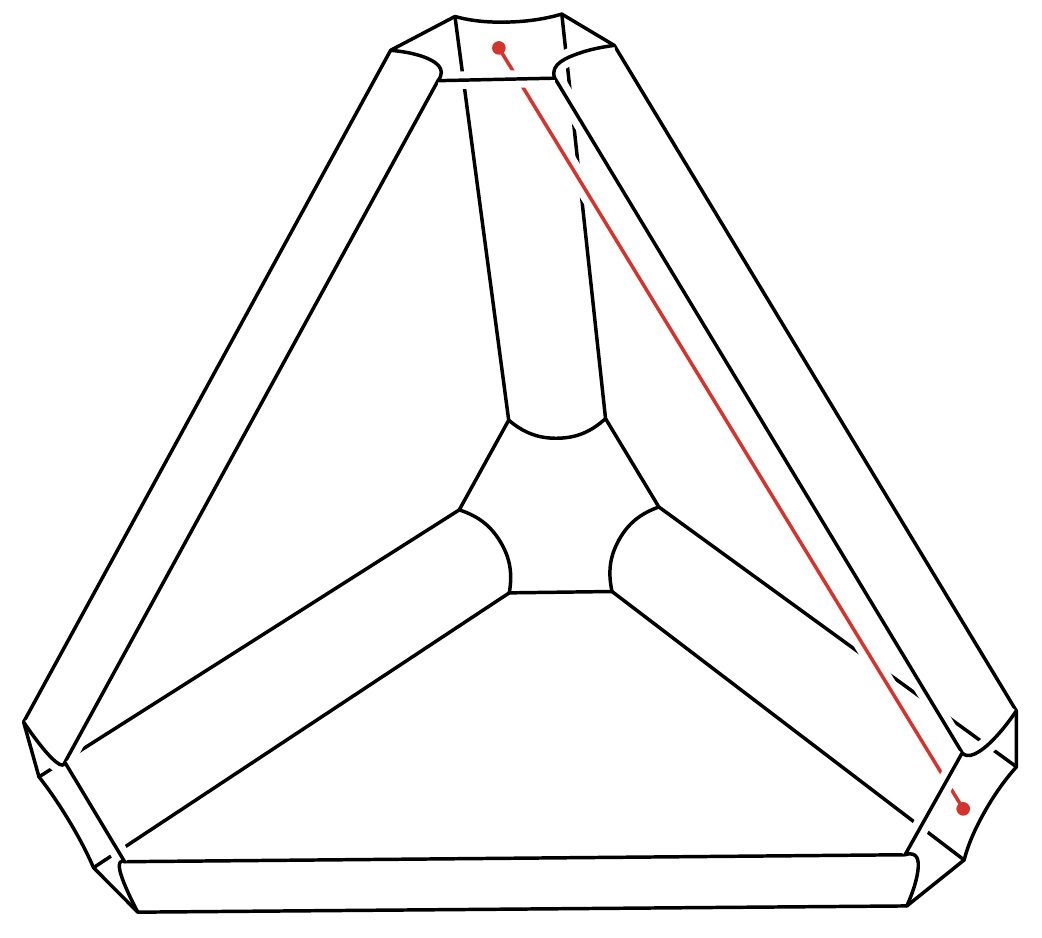}
\caption{A doubly truncated tetrahedron of $H,$ with the path $\alpha$ parallel to the boundary rectangle corresponding to an edge of $P$.}
\label{trunc_tetra.pdf}
\end{figure}

Since $e$ is inessential in $C$, there is a homotopy of it into $\bdry C$, fixing its endpoints. As $\alpha$ (viewed as a path in $C$ since $H\subset C$) is parallel to $e$, we can use the same homotopy to homotope $\alpha$ into $\bdry C$, fixing its endpoints, by first homotoping across the rectangle with $\alpha$ and $e$ as one pair of opposite sides, and the other sides on $\bdry C$ in the obvious way. Viewing the homotopy as a map $D^2 = \bigcirc\rightarrow C$, we may deform it by a small amount to produce a homotopy $h$ transverse to $P^{(1)}$, and moreover so that the intersection of $h(\bigcirc)$ with $\bdry (\nu(P^{(1)})$ consists of a finite number of circles, each of which (pulled back through $h$) bounds a disk in $\bigcirc$ which is contained in a neighbourhood $\nu(e_i)$ and intersects $e_i$ once, transversely. Thus each circle goes around the cylindrical part of $\bdry C$ corresponding to an edge of $N$, and is homotopic to a $\gamma_e$. In particular let $\beta:[0,1]\rightarrow \bdry C$ be the path after the homotopy, and by deforming if necessary we may assume that $\beta$ in fact maps into $\bdry H \cap \bdry C$. See Figure \ref{homotopy_disk.pdf}.

\begin{figure}[ht]
\centering
\includegraphics[width=0.6\textwidth]{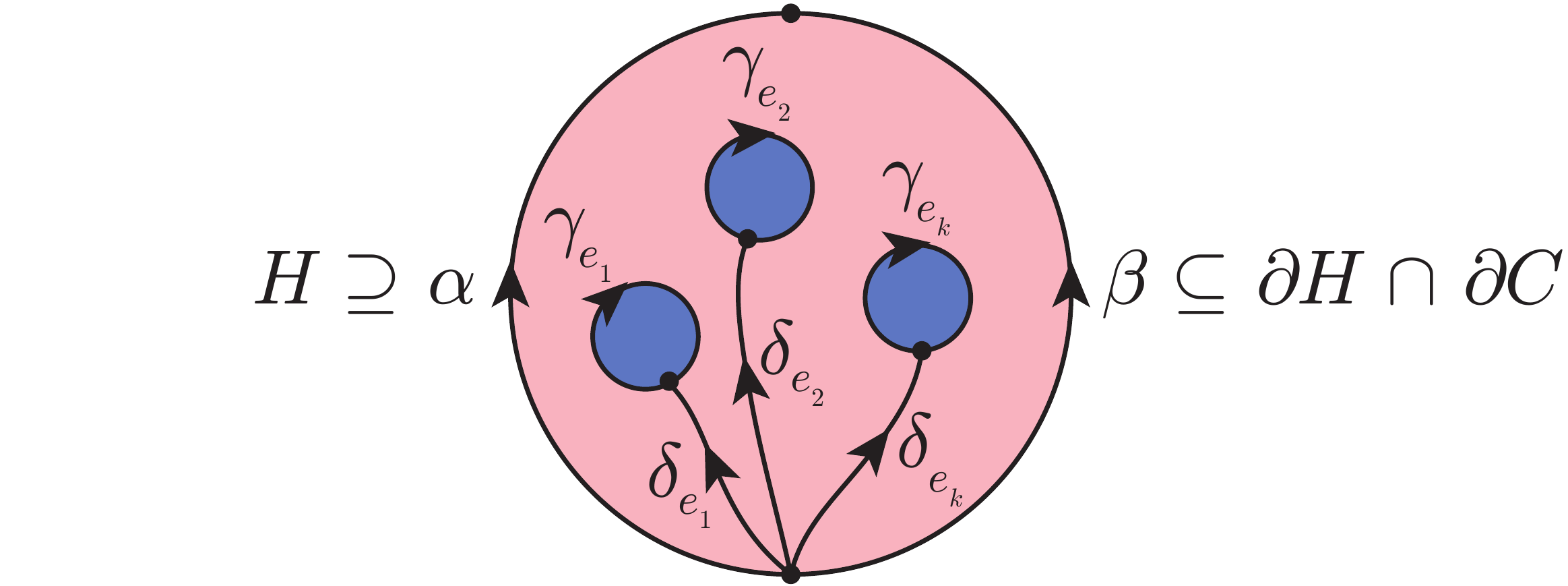}
\caption{The homotopy $h:\bigcirc\rightarrow C$ between $\alpha$ and $\beta$. The parts mapping into $H$ are shaded lightly, and the disks mapping into the neighbourhood of an edge in $C$ shaded more darkly. }
\label{homotopy_disk.pdf}
\end{figure}

We choose a base-point for $\pi_1(H)$ at $\alpha(0).$ Now $\bigcirc\cap h^{-1}(H)$ is a sphere with a number of boundary components, where $\partial \bigcirc = \alpha \beta^{-1}$ and all other boundary components correspond to circles $\gamma_k\co [0,1] \to H,$ which are fibres in the product structure of the edge boundary. For each $\gamma_k$ choose a non self-intersecting path $\delta_k$ on $\bigcirc\cap h^{-1}(H)$ from $\alpha(0)$ to $\gamma_k(0)= \gamma_k(1).$ We can choose these paths to be disjoint apart from at $\alpha(0)$. Then $\bigcirc\cap h^{-1}(H)$ gives the following relation between the elements in $\pi_1(H, \alpha(0))$ represented by these loops:
\begin{equation}\label{eq:disc relator}
[\alpha\beta^{-1}] = \prod [\delta_k\gamma_k \delta_k^{-1}].
\end{equation}
Choose a lift of $\alpha(0)$ to $\til{H}$ and lift $\alpha$ accordingly. Let $\til{\alpha}^*$ be the composition with the quotient $\til{H}\rightarrow\til{H}^*.$ Then $D\til{\alpha}^*:[0,1]\rightarrow  \overline{\mathbb{H}}^3$ is its image under the pseudo-developing map. Define $\til{\beta}$, $\til{\beta}^*$ and $D\til{\beta}^*$ similarly, choosing the lift $\til{\beta}$ to start from the same point as $\til{\alpha}$ does. Since the pseudo-developing map is well-defined and since the endpoints of $e$ are on the same components of $\bdry H \cap \bdry C$ as the respective ends of $\alpha$, $D\til{\alpha}^*(0)$ and $D\til{\alpha}^*(1)$ are distinct points on $\bdry \Hthree$. With our choice of base-points, $\rho([\alpha\beta^{-1}])$ is an isometry of $\overline{\mathbb{H}}^3$ taking $D\til{\alpha}^*(0)$ to $D\til{\alpha}^*(1).$

We now claim that for each term $\delta_k\gamma_k \delta_k^{-1}$ in the right hand side of relation (\ref{eq:disc relator}), we have $\rho([\delta_k\gamma_k \delta_k^{-1}])=1.$ Indeed, $\rho([\delta_k\gamma_k \delta_k^{-1}])$ can be expressed as a product of elementary face pairings. Now the products arising from $\delta_k$ and $\delta_k^{-1}$ are inverses. Hence $\rho([\delta_k\gamma_k \delta_k^{-1}])=1$ if and only if the product corresponding to $\gamma_k$ is trivial. But $\gamma_k$ gives a rotation with eigenvalue the product of all shape parameters around the edge corresponding to that loop. By hypothesis, this product equals $1,$ and hence the claim. 

But then $\rho([\alpha\beta^{-1}])=1,$ contradicting the fact that it acts non-trivially on $\partial \H^3.$ This completes the proof of Theorem~\ref{thm:essential edges}.

For the general case, the above proof can be applied using a punctured sphere in $N_o.$ The curves $\gamma_k$ can only correspond to edges of $N_o$ which have finite degree, and the eigenvalue of the associated rotation is precisely of the form $\xi_e^{o_e}=1,$ giving the same contradiction as above.
\end{proof}

\begin{rmk}
The main restriction in using Yoshida's construction is the fact that ideal simplices in the universal cover are mapped to hyperbolic ideal simplices in hyperbolic 3--space. Daryl Cooper pointed out to us that the proofs in this section could be given using local arguments by subdividing the ideal tetrahedra in $N.$ The authors feel that the approach using triangulations is more appropriate for this volume and the man it is dedicated to.
\end{rmk}



\section{A trip to the zoo}
\label{sec:zoo}

We give three examples of triangulations of $S^3$ with knotted or linked edges, which exhibit interesting features. Special hyperbolic cone-manifold structures on the first triangulation can be found in work by Boileau-Porti \cite{bp} and Hodgson \cite{cdh}, and the two concluding triangulations were provided by Bus Jaco.

\begin{figure}[htb]
\centering
\includegraphics[width=0.9\textwidth]{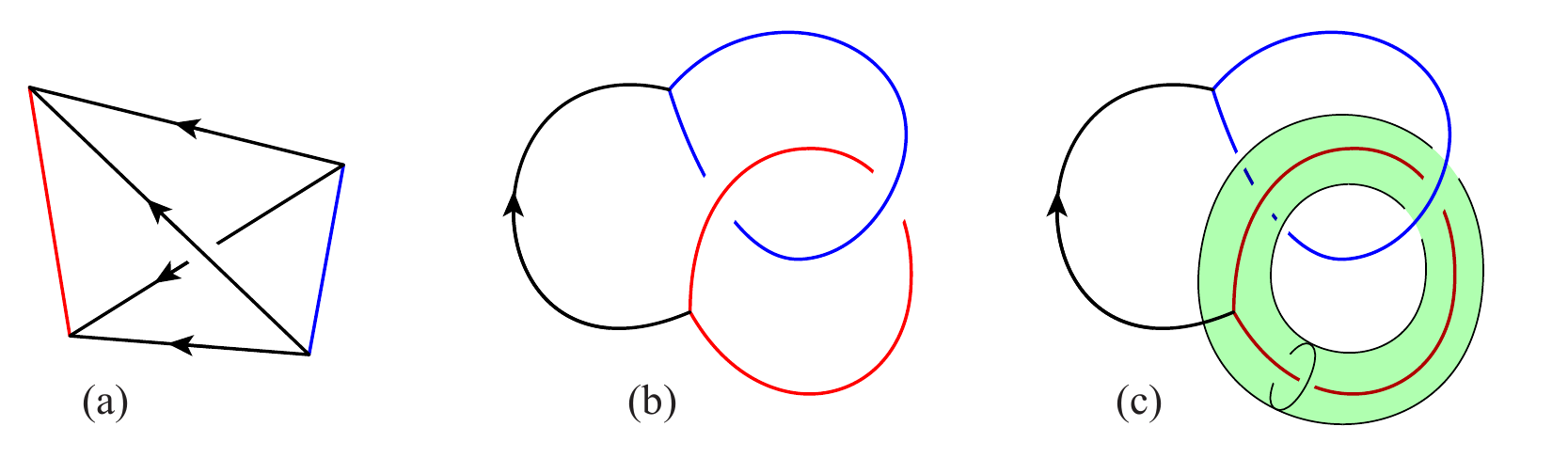}
\caption{Triangulation of $S^3$ with the Hopf link as edges.}
\label{bph_cone_mfld.pdf}
\end{figure}


\subsection{The Hopf link}

Consider the one-tetrahedron three-edges triangulation of $S^3$ shown in Figure~\ref{bph_cone_mfld.pdf}(a). We have $\D(\tri) = \emptyset,$ since there are degree one edges. Giving the parameter $z$ to the degree-one edges $e_0$ and $e_1$, and letting $e_2$ be the degree-four edge, the cone-deformation variety is:
$$\D(\tri; \star) = \{ (Z, \xi) = ((z, z', z''), (z, z, z^{-2})) \mid z \in S^1 \setminus \{1\} \ \} \cong S^1 \setminus \{1\}.$$

There is one ideal point (corresponding to $z\to 1$) and one flat solution (corresponding to $z=-1$). The degeneration $z\to 1$ corresponds to a normal surface which is a Heegaard torus in $S^3,$ and the flat solution will be analysed below using the face pairings.

The remaining structures come in pairs $(z, \overline{z})$ and it suffices to study the case $\Im(z)>0.$ Here, $|z|=1$ implies that the arguments of $z, z', z''$ are the angles of an isoceles triangle. Letting $\alpha$ denote the argument of $z,$ the angle around $e_2$ is $2(\pi-\alpha)$ and the angle around each of $e_0$ and $e_1$ is $\alpha.$ This gives hyperbolic cone-manifold structures, with respective cone angles $(\alpha, \alpha, 2(\pi-\alpha)).$ Since the parameter $z$ at the degree-one edges is an element of $S^1,$ it is easily verified that all of these hyperbolic cone-manifold structures are complete; see Figure~\ref{tet_spine_horospheres1.pdf}.

\begin{figure}[htb]
\centering
\includegraphics[width=0.9\textwidth]{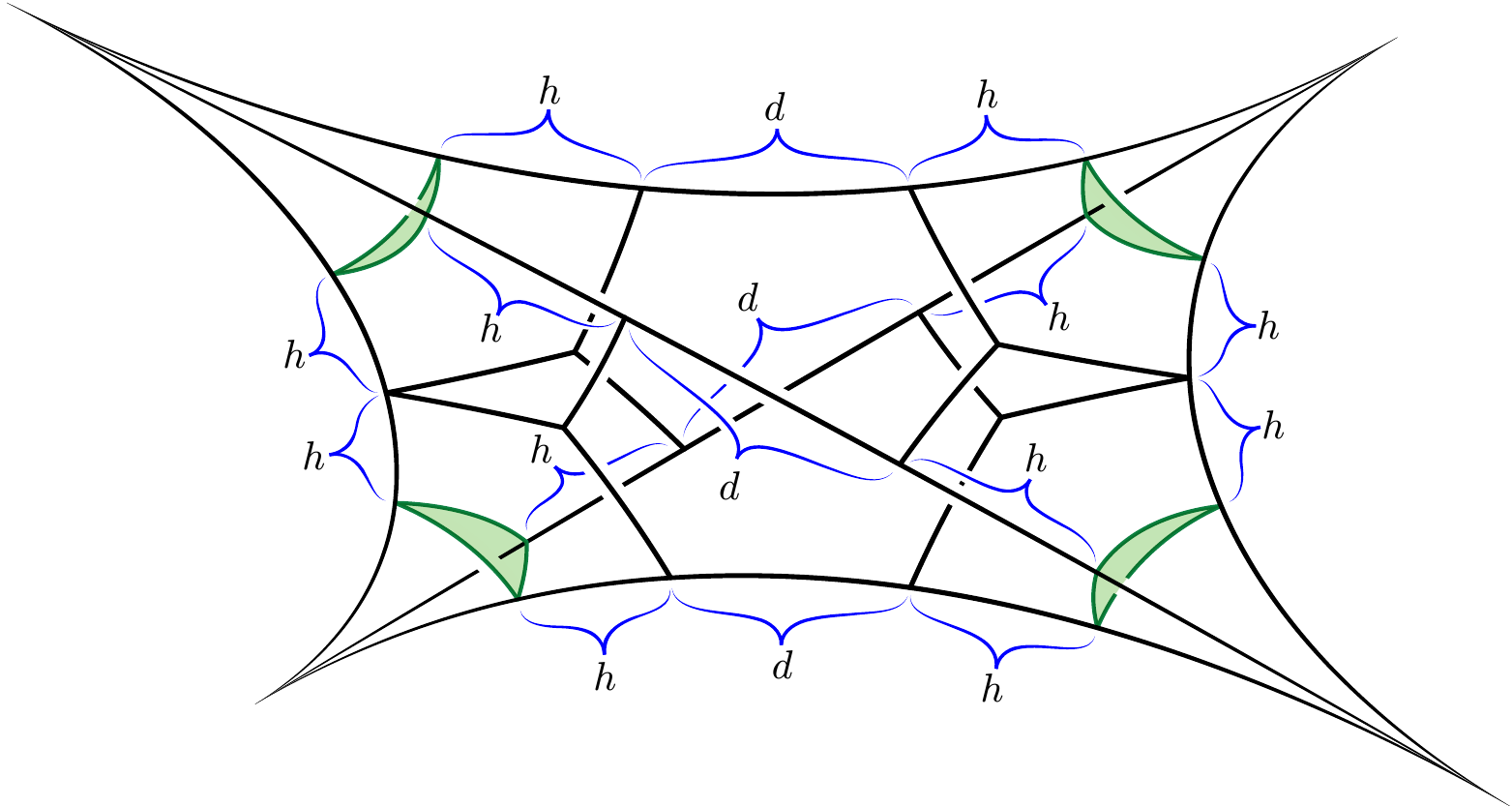}
\caption{Geometry of the hyperbolic cone-manifolds for the Hopf link: The shown horospherical triangles give a cross-section of the cusps. The two face pairings are rotations about the left and right edges respectively.}
\label{tet_spine_horospheres1.pdf}
\end{figure}

In order to analyse the face pairings, we normalise the developing map so that the degree one edges are mapped to the geodesics $[1,\infty]$ and $[0, z''].$ The corresponding pairings for the faces incident with these edges are denoted $\gamma_0$ and $\gamma_1$ respectively. The holonomy around the third edge is denoted $\gamma_2.$ The images of the group elements under the holonomy representation are determined by the following M\"obius transformations:
\begin{align*}
&\rho_z(\gamma_0):[1,\infty, z'']\mapsto[1,\infty,0], \\
&\rho_z(\gamma_1):[0, z'', 1]\mapsto[0,z'',\infty], \\
&\rho_z(\gamma_2):[0, \infty, 1]\mapsto[0,\infty,z^{-2}],
\end{align*}
where $\rho_z(\gamma_2)=\rho_z([\gamma_1, \gamma_0]).$ This gives:
\begin{equation*}
\rho_z(\gamma_0) = \frac{1}{\sqrt{z}} \left( \begin{array}{cc} z & 1-z\\ 0 & 1\end{array} \right),\qquad
\rho_z(\gamma_1) = \frac{1}{\sqrt{z}} \left( \begin{array}{cc} 1 & 0\\ -z & z\end{array} \right),\qquad
\rho_z(\gamma_2) =  \left( \begin{array}{cc} z & 0\\ 0 & z^{-1}\end{array} \right).
\end{equation*}
Letting $z=e^{i\theta},$ we get the traces $(2 \cos(\theta/2), 2 \cos(\theta/2), e^{i\theta}+e^{-i\theta}).$

At $z = -1,$ we have
\begin{equation*}
\rho_{-1}(\gamma_0) = \left( \begin{array}{cc} i & -2i\\ 0 & -i\end{array} \right),\qquad
\rho_{-1}(\gamma_1) = \left( \begin{array}{cc} -i & 0\\ -i & i\end{array} \right),\qquad
\rho_{-1}(\gamma_2) = \left( \begin{array}{cc} -1 & 0\\ 0 & -1\end{array} \right). 
\end{equation*}

Geometrically, as $z\to -1,$ the fundamental domain degenerates to a quadrilateral, and the identification space is a sphere with cone angles $(0,0,\pi)$. By construction, this is a hyperbolic 2--cone-manifold.

As $z \rightarrow 1,$ the limiting representation is infinite cyclic:
\begin{equation*}
\rho_{1}(\gamma_0) = \left( \begin{array}{cc} 1 & 0\\ 0 & 1\end{array} \right),\qquad
\rho_{1}(\gamma_1) = \left( \begin{array}{cc} 1 & 0\\ -1 & 1\end{array} \right).
\end{equation*}
The fact that $\rho_{1}(\gamma_1)$ is the generator of the image corresponds to the chosen normalisation: The edge $[0, z'']$ pops off at infinity (since $z''\to 0$ and $z\to 1$), giving $\rho_{1}(\gamma_0)=1.$


\subsection{The trefoil knot} 

Let $\tri_{3_1}$ denote the minimal layered triangulation of $S^3$ with one degree-one edge $e_1$ and one degree-five edge $e_2,$ as shown in Figure \ref{knot_triangs_of_S3.pdf}. Here, $e_1$ is the trefoil knot. Again, $\D(\tri_{3_1}) = \emptyset$ and
$$\D(\tri_{3_1}; \star) = \{ (Z, \xi) = ((z, z', z''), (z, z^{-1})) \mid z \in S^1 \setminus \{1\} \ \} \cong S^1 \setminus \{1\},$$
where $z$ is the parameter given to $e_1.$ The angle at $e_1$ is $\alpha$ and the angle at $e_2$ is $2\pi-\alpha.$

\begin{figure}[htb]
\centering
\includegraphics[width=0.9\textwidth]{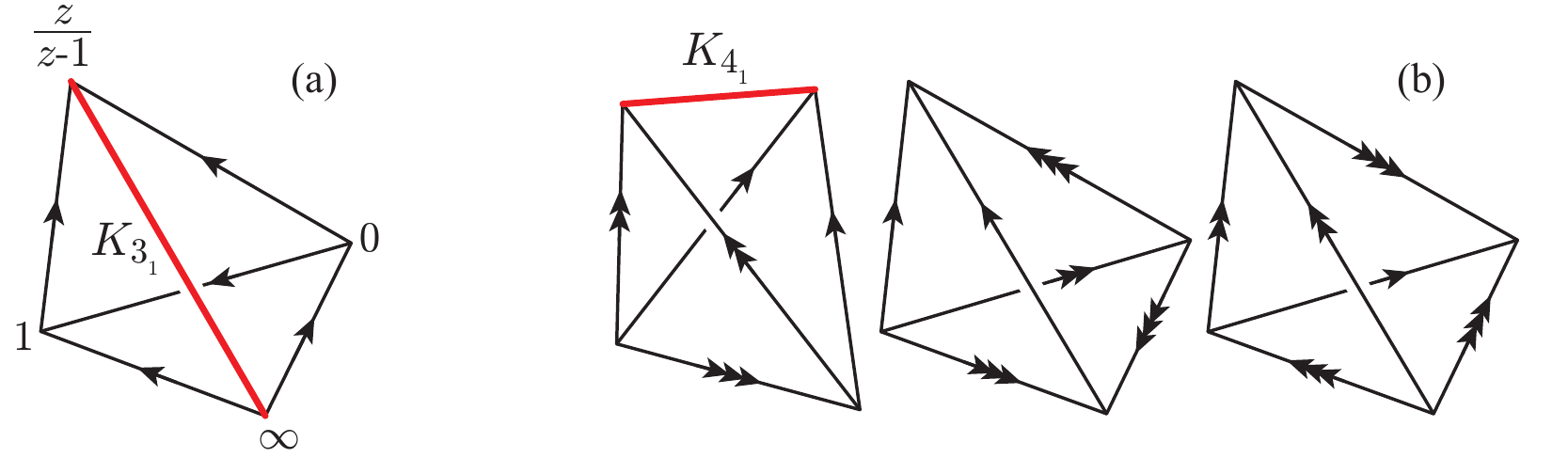}
\caption{Triangulations of $S^3$ with knots as edges. In diagram (a) the edge $K_{3_1}$ is a trefoil. In diagram (b) the edge $K_{4_1}$ is the figure 8 knot.}
\label{knot_triangs_of_S3.pdf}
\end{figure}

As above, there is an ideal point (corresponding to $z\to 1$), and one flat solution (corresponding to $z=-1$). They will be analysed below using the face pairings. The degeneration $z\to 1$ corresponds to a thin edge-linking torus for $e_K,$ and hence to the trefoil knot complement. 

The remaining structures come in pairs $(z, \overline{z})$ and it can again be verified directly that they all give complete hyperbolic cone-manifold structures with singular locus consisting of the three edges; see Figure~\ref{tet_spine_horospheres2.pdf}.

\begin{figure}[htb]
\centering
\includegraphics[width=0.9\textwidth]{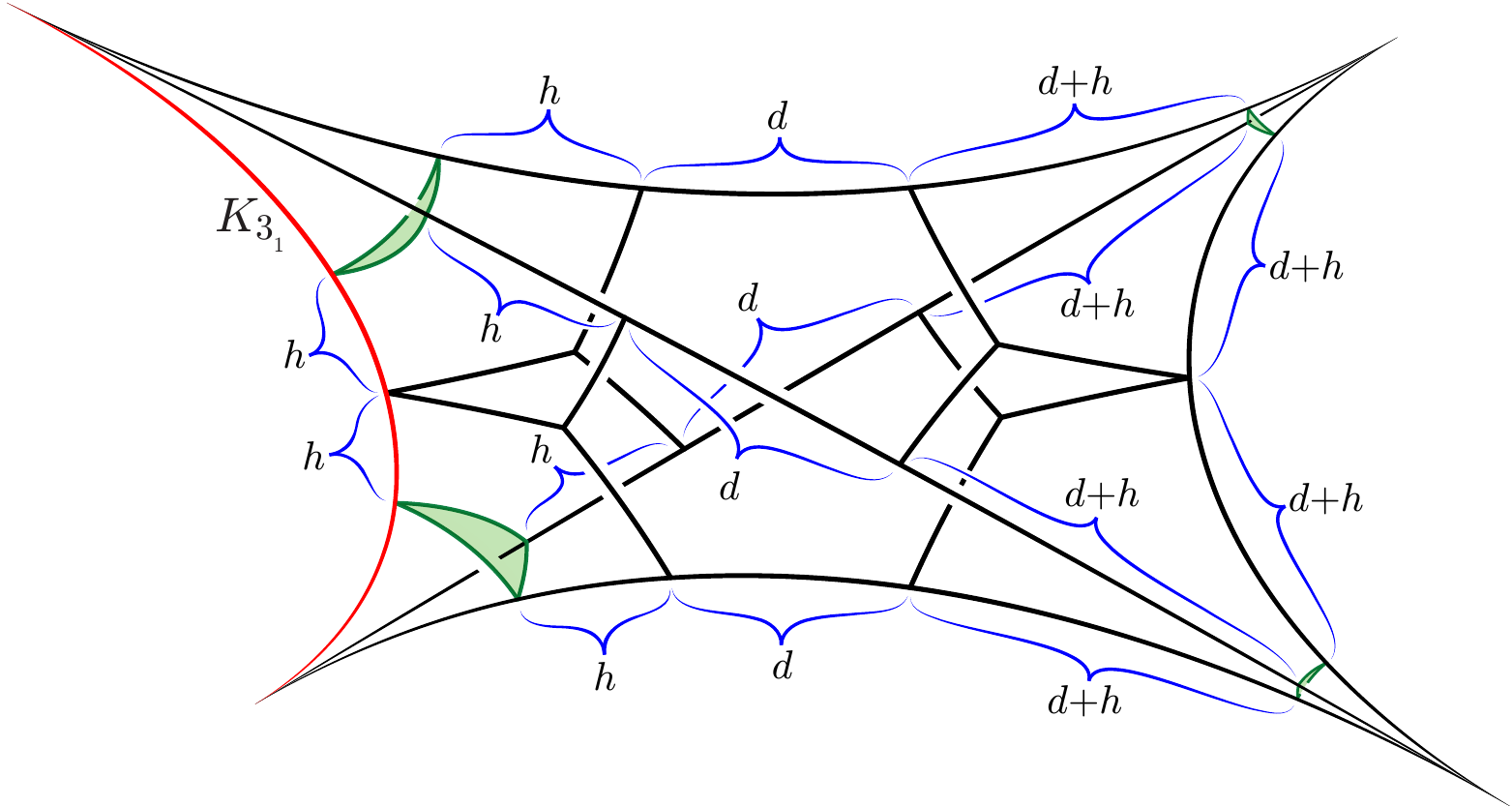}
\caption{Geometry of the hyperbolic cone-manifolds for the trefoil: The shown horospherical triangles give a cross-section of the cusps. One face pairing is a rotation about the left-hand edge; the other identifies the faces incident with the right-hand edge with a twist.}
\label{tet_spine_horospheres2.pdf}
\end{figure}

The fundamental group is again generated by two face pairings, denoted $\gamma_3$ and $\gamma_\infty.$ The holonomy images are:

$$\rho_{z}(\gamma_3):[0,1,\infty]\mapsto[1,\frac{z}{z-1},0], \qquad
\rho_{z}(\gamma_\infty):[\infty, \frac{z}{z-1}, 0]\mapsto[\infty, \frac{z}{z-1},1].
$$

These lead to representations: 

\begin{equation*}
\rho_{z}(\gamma_3) = \frac{1}{\sqrt{z}}\left( \begin{array}{cc} 0 & z\\ -1 & z\end{array} \right),\qquad
\rho_{z}(\gamma_\infty) = \frac{1}{\sqrt{z}} \left( \begin{array}{cc} 1 & z\\ 0 & z\end{array} \right).
\end{equation*}
Their product $\gamma_2 = \gamma_3 \gamma_\infty$ is always of order two:
$$\rho_{z}(\gamma_2) =  \left( \begin{array}{cc} 0 & z\\ -z^{-1} & 0\end{array} \right).$$
At $z\rightarrow 1$, we have: 
\begin{equation*}
\rho_{1}(\gamma_3) = \left( \begin{array}{cc} 0 & 1\\ -1 & 1\end{array} \right),\qquad
\rho_{1}(\gamma_\infty) = \left( \begin{array}{cc} 1 & 1\\ 0 & 1\end{array} \right),\qquad
\rho_{1}(\gamma_2) =  \left( \begin{array}{cc} 0 & 1\\ -1 & 0\end{array} \right).
\end{equation*}

These generate the modular group, and the quotient of $\H^2$ is a sphere with 2, 3 and $\infty$ cone points. See Figure \ref{trefoil_collapse.pdf}. The flat solution does not appear to have a nice interpretation.

 \begin{figure}[htb]
\centering
\includegraphics[width=0.5\textwidth]{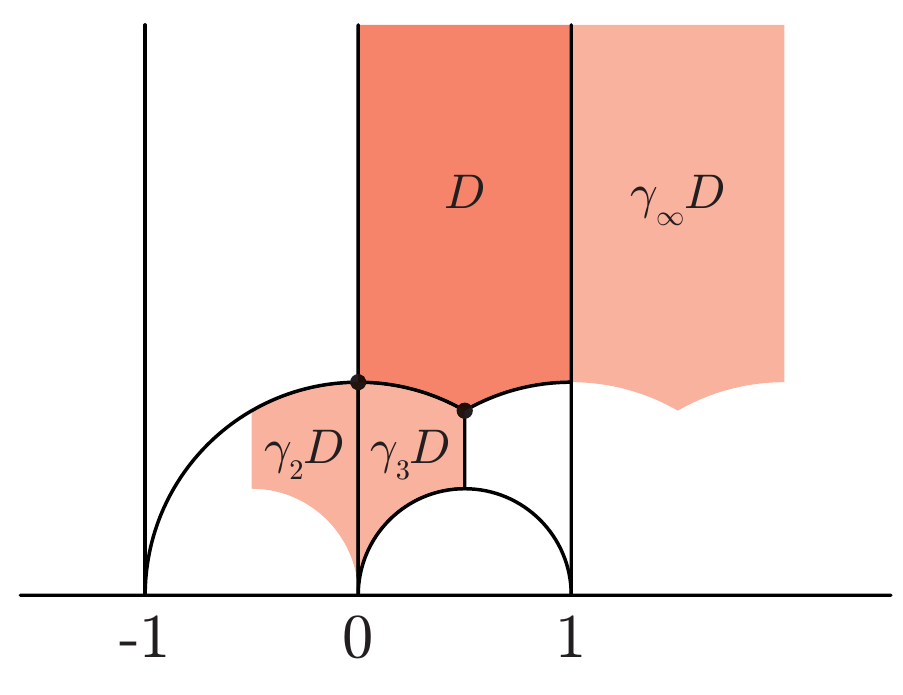}
\caption{A fundamental domain of the action of $\rho_1$ for the trefoil on $\H^2$ is marked $D$.}
\label{trefoil_collapse.pdf}
\end{figure}


\subsection{The figure eight knot}

In the previous examples, we have analysed structures arising at ordinary points of the cone-deformation variety as well as a single ideal point which corresponding to a 1--dimensional or 2--dimensional collapse. In the last example, we exhibit a surface of ideal points parameterising 3--dimensional structures. See Figure \ref{knot_triangs_of_S3.pdf}(b).

Label the tetrahedra $z_0,z_1$ and $z_2$ at their upper and lower edges in Figure~\ref{knot_triangs_of_S3.pdf}. Then the generalised gluing equations are: 
\begin{eqnarray}
z_0 & = & \xi_K\\
\left(\frac{1}{z_0} \right)\left( \frac{1}{z_1} \right)\left( \frac{z_2-1}{z_2}\right) & = & \xi_1\\
\left(\frac{1}{z_0} \right)\left( \frac{z_1-1}{z_1} \right)\left( \frac{1}{z_2}\right) & = & \xi_2\\
z_0 \left( \frac{z_1^2}{1-z_1} \right)\left( \frac{z_2^2}{1-z_2} \right) & = & \xi_3
\end{eqnarray}
These simplify to:
\begin{eqnarray}
z_0 & = & \xi_K\\
\xi_3 & = & \xi_K^{-1} \xi_1^{-1}  \xi_2^{-1}  \\
z_1 z_2 \xi_2 \xi_K  & = &  (z_1 -1)\\
z_1 z_2 \xi_1 \xi_K  & = &  (z_2 -1)
\end{eqnarray}
Hence $\D(\tri_{4_1}; \star)$ is parameterised by $(\xi_K, z_1, z_2),$ where $\xi_K \in S^1$ and $z_1, z_2 \in \C\setminus\{0,1\},$ subject to $\frac{z_2''}{z_1}\in S^1$ from (12) and $\frac{z_1''}{z_2} \in S^1$ from (11) (note that $z_i''=\frac{z_i-1}{z_i}).$

Putting $z_1 = e^{-i\theta}z_2'',$ where $\theta \in [0,2\pi)$, the second condition gives $z_2'e^{i\theta}-z_2''z_2' \in S^1$. Letting $z_2''=x+iy$ for $x,y\in\RR$, this yields a single equation in $x, y, \theta,$ and the space of solutions is a surface in $\RR^2 \times S^1$ as shown in Figure \ref{fig8_edge_in_s3_real_param_space2.jpg}. Now $\D(\tri_{4_1}; \star)$ is parameterised by $(\xi_K, x, y, \theta),$ and since $\xi_K \in S^1$ is arbitrary, $\D(\tri; \star)$ is the product of this surface with $S^1$.

 \begin{figure}[htb]
\centering
\includegraphics[width=0.8\textwidth]{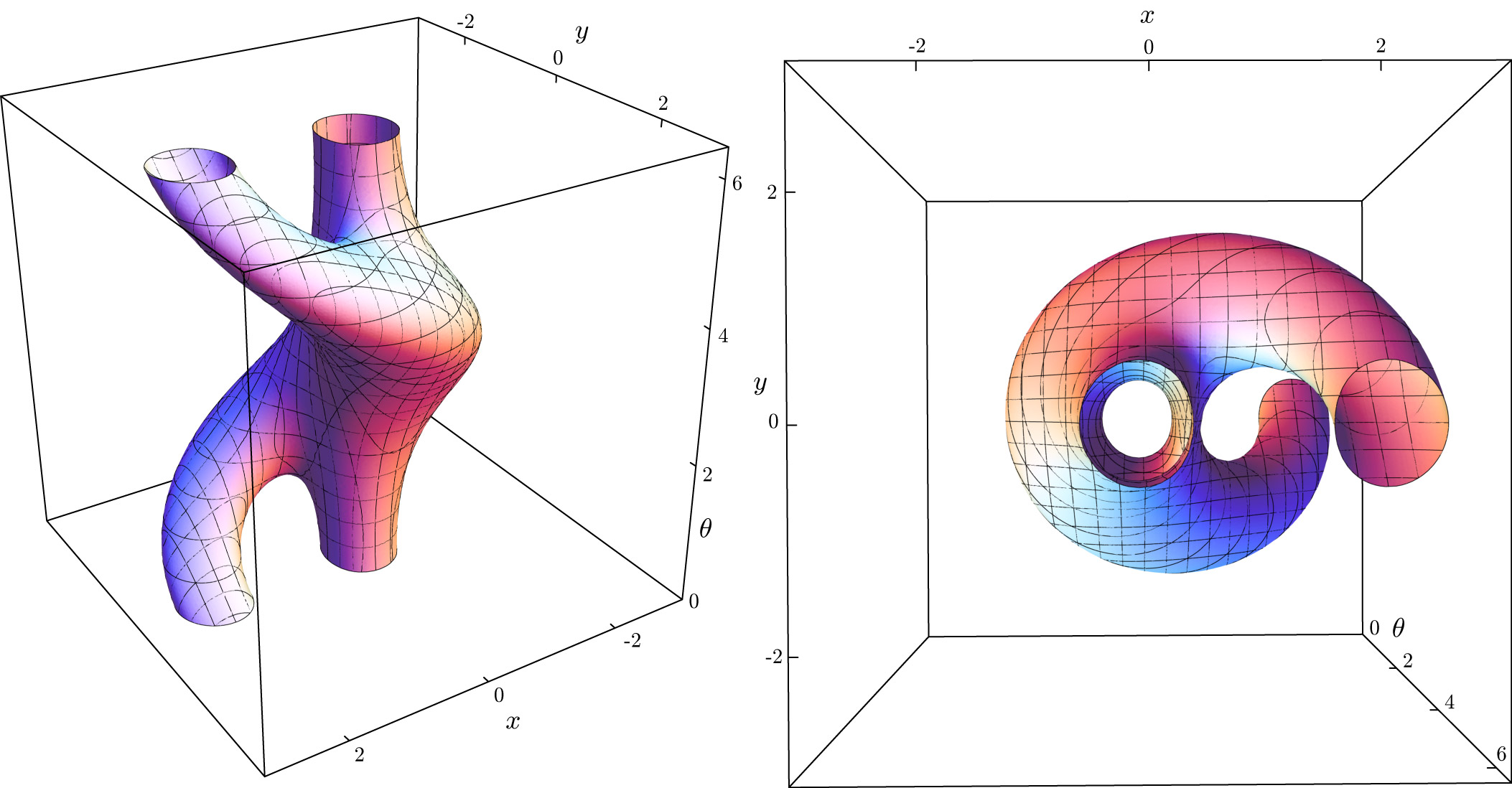}
\caption{$\D(\tri_{4_1}; \star)$ is the product of this surface in $\RR^2 \times S^1$ with $S^1$.}
\label{fig8_edge_in_s3_real_param_space2.jpg}
\end{figure}

The only ideal points are at $\xi_K = 1$, parameterised by $(1, x, y, \theta),$ and the associated normal surface is the thin edge linking torus around $K$, which splits the manifold into a solid torus and the figure 8 knot complement. Crushing along this surface identifies the single arrow edge with the double arrow edge, and gives us the canonical triangulation of the figure 8 knot complement. The equations become:

\begin{eqnarray}
z_0 & = & 1\\
 \frac{z_2-1}{z_1 z_2} & = & \xi_1\\
 \frac{z_1-1}{z_1z_2}  & = & \xi_2\\
\frac{z_1^2z_2^2}{(1-z_1)(1-z_2)} & = & \xi_3
\end{eqnarray}

To get the complete structure on the figure eight knot complement, we need $\xi_1=1=\xi_2$, which implies that $z_1=z_2''$ and $z_2=z_1''$. Then $z_1 = e^{-i\theta}z_2'' = e^{-i\theta}z_1$, so $\theta=0$, and we get $z_1 = z_2''$. We also have $z_2=z_1''$, which implies that $z_2''=z_1'$, so $z_1=z_1'$, and $z_1$ is the shape parameter of the regular hyperbolic ideal tetrahedron. Similarly for $z_2$. Note that solutions in a neighbourhood of this solution give 3--dimensional hyperbolic cone-manifold structures on the figure eight knot complement with singular locus contained in the ideal edges.



\address{Henry Segerman,\\ Department of Mathematics and Statistics,\\ The University of Melbourne\\ VIC 3010, Australia\\
(segerman@unimelb.edu.au)\\--}

\address{Stephan Tillmann,\\ School of Mathematics and Physics,\\ The University of Queensland,\\ Brisbane, QLD 4072, Australia\\
(tillmann@maths.uq.edu.au)}

\Addresses



\begin{thebibliography}{99}

\bibitem{bp} Michel Boileau, Joan Porti: \emph{Geometrization of 3-orbifolds of cyclic type}, Ast\'erisque 272, (2001).

\bibitem{ac} Abhijit Champanerkar: \emph{A-polynomial and Bloch invariants of hyperbolic 3-manifolds}, preprint.

\bibitem{ch} Young-Eun Choi: \emph{Positively oriented ideal triangulations on hyperbolic three-manifolds}, Topology 43, no. 6, 1345--1371 (2004).


\bibitem{f}  Stefano Francaviglia: \emph{Hyperbolic volume of representations of
fundamental groups of cusped 3-manifolds}, IMRN, no.9,425-459 (2004).

\bibitem{cdh} Craig D. Hodgson: \emph{Degeneration and regeneration of hyperbolic structures on three-manifolds}, Ph. D. thesis, Princeton University, 1986.    

\bibitem{luo} Feng Luo: \emph{Volume optimization, normal surfaces and Thurston's equation on triangulated 3-manifolds}, arXiv:0903.1138v3.

\bibitem{LTY} Feng Luo, Stephan Tillmann and Tian Yang: \emph{Thurston's spinning construction and solutions to the hyperbolic gluing equations for closed hyperbolic 3-manifolds}, 	Proceedings of the American Mathematical Society, in press; arXiv:1004.2992v1.

\bibitem{nr} Walter D. Neumann, Alan W. Reid: \emph{Arithmetic of hyperbolic manifolds},  Topology '90 (Columbus, OH, 1990), 273--310, Ohio State Univ. Math. Res. Inst. Publ., 1, de Gruyter, Berlin, 1992.
  
\bibitem{nz} Walter D. Neumann, D. Zagier: \emph{Volumes of hyperbolic three--manifolds}, Topology, 24, 307-332 (1985).

\bibitem{ny} Walter D. Neumann, Jun Yang: \emph{Bloch invariants of hyperbolic 3-manifolds}, Duke Math. J., 96, no. 1 (1999), 29-59.

\bibitem{seg} Henry Segerman: \emph{Detection of incompressible surfaces in hyperbolic punctured torus bundles}, Geometriae Dedicata, no. 1, 150 (2011), 181-232.
    
\bibitem{t} William P. Thurston: \emph{The geometry and topology of 3--manifolds}, Princeton Univ. Math. Dept. (1978). Available from http://msri.org/publications/books/gt3m/.

\bibitem{t1} William P. Thurston: \emph{Hyperbolic Structures on 3--Manifolds I: Deformation of Acylindrical manifolds}, Ann. of Math., 124, 203-246 (1986).

\bibitem{tillus_normal} Stephan Tillmann: \emph{Normal surfaces in topologically finite 3-manifolds}, Enseign. Math. (2) 54 (2008), no. 3-4, 329-380.

\bibitem{ti} Stephan Tillmann: \emph{Degenerations of ideal hyperbolic triangulations}, Mathematische Zeitschrift, in press; arXiv:math/0508295v3

\bibitem{y} Tomoyoshi Yoshida: \emph{On ideal points of deformation curves of hyperbolic 3--manifolds with one cusp}, Topology, 30, 155-170 (1991).

\end{thebibliography}
\end{document}